\documentclass[twocolumn,amsthm,authoryear]{autart}    

\usepackage{graphicx}          
\usepackage{amsmath,amssymb,amsfonts} 
\usepackage{tikz}

\usetikzlibrary{shapes,arrows,positioning}

\tikzset{
	block/.style={
		draw, 
		fill=blue!10, 
		rectangle, 
		minimum height=3em, 
		minimum width=6em
	},
	sum/.style={
		draw, 
		fill=blue!10, 
		circle, minimum size= .1em,
	},
	input/.style={coordinate},
	output/.style={coordinate},
	pinstyle/.style={
		pin edge={to-,thin,black}
	}
}  
\usepackage{natbib}
\usepackage{xcolor, hyperref}

\definecolor{darkblue}{rgb}{0.0,0.0,0.6}
\hypersetup{colorlinks,breaklinks,linkcolor=darkblue,urlcolor=darkblue,anchorcolor=darkblue,citecolor=darkblue}

\theoremstyle{plain}
\newtheorem{theorem}{Theorem}[section]

\newtheorem{corollary}[theorem]{Corollary}
\newtheorem{lemma}[theorem]{Lemma}
\newtheorem{proposition}[theorem]{Proposition}

\theoremstyle{definition}
\newtheorem{definition}[theorem]{Definition}

\newtheorem{problem}[theorem]{Problem}

\theoremstyle{remark}
\newtheorem{remark}[theorem]{Remark}

\makeatletter
\def\endthebibliography{%
	\def\@noitemerr{\@latex@warning{Empty `thebibliography' environment}}%
	\endlist
}
\makeatother

\begin{document}

\begin{frontmatter}

\title{On Control of Drawdown: Robust Invariance and Optimality} 


\thanks[footnoteinfo]{
	This paper is partially supported by National Science and Technology Council with grant: NSTC115-2628-E-007-005-. 
	}

\author[CHHSIEH]{Chung-Han Hsieh }\ead{ch.hsieh@mx.nthu.edu.tw} 

\address[CHHSIEH]{Department of Quantitative Finance, National Tsing Hua University, Hsinchu, Taiwan 30013}  


\begin{abstract}  
Mitigating \emph{drawdown}, the decline in wealth from its running peak, presents a canonical problem in path-dependent risk control. 
In this paper, we develop a finite-horizon control framework that enforces a prescribed maximum percentage drawdown limit in multi-asset stochastic systems. 
Our first result is an exact robust-invariance theorem characterizing every control action that preserves a prescribed drawdown limit against all supported returns.
We show that every robustly safe control admits a \emph{drawdown-modulated} form: the product of the current drawdown \emph{cushion} and a feasible \emph{normalized direction}. 
This yields a complete parameterization of robustly drawdown-safe policies.
Additionally, under stagewise-independent returns, we show that optimizing over all robustly safe causal policies reduces to a one-dimensional Bellman recursion and yields an optimal robustly safe state-feedback policy. 
Finally, we characterize the linear time-invariant (LTI) gains satisfying a prescribed drawdown limit and prove that optimal drawdown modulation achieves no lower expected return under the same limit.
Strict expected-return improvement holds for horizons of at least two stages whenever the LTI policy has positive expected one-stage net return.
\end{abstract}

\end{frontmatter}
\endNoHyper 

\section{Introduction}
\label{INTRODUCTION}

\emph{Drawdown} is the decline of wealth from its running maximum and
is therefore inherently path dependent. 
We study finite-horizon control of maximum percentage drawdown in a multi-asset stochastic system. 
Our objective is to characterize and optimize control policies that
enforce a prescribed drawdown limit against every supported
return path.

A natural benchmark is the linear feedback law $u(k)=KV(k)$, under which the risky-asset exposures are fixed fractions of current account value. 
This proportional-to-wealth structure is closely related to the fixed-weight strategies arising in classical mean--variance \cite{Markowitz_1952,Malekpour_Barmish_2012} and Kelly betting formulations; see
\cite{Kelly_1956,maclean2011kelly,hsieh2016kelly}. 
It is also central to the control-based investment models studied in
\cite{calafiore2008multi,Barmish_Primbs_2015,hsieh2023asymptotic,barmish2024jump}.

\subsection{Drawdown Safety and Related Work}

In classical mean--variance analysis, the return--risk pair
consists of expected return and return variance; see
\cite{Markowitz_1952,luenberger2013investment}.
Although tractable and widely used, variance treats favorable and unfavorable deviations symmetrically and does not directly measure losses relative to a running wealth maximum; see \cite{konno1991mean,hsieh2017inefficiency}.
Thus, drawdown-based risk measures and constraints have received
considerable attention; see
\cite{Ismail_2004,chekhlov2004portfolio,chekhlov2005drawdown,
grossman1993optimal,cvitanic1994portfolio,nystrup2019multi, calafiore2013direct}.
Here, we study a discrete-time stochastic control problem under a prescribed drawdown limit that must hold along every supported return path.  
We characterize all control actions that preserve the drawdown limit and optimize expected terminal wealth over the resulting class of robustly safe causal policies.

Our drawdown-modulated control class has a \emph{cushion-based} structure: its exposure is proportional to the amount by which current account value exceeds a floor linked to the running maximum. 
The cushion form itself has continuous-time antecedents. 
For example, \cite{grossman1993optimal} show that the optimal risky
exposure is proportional to wealth above a floor linked to its running
maximum. 
More general continuous-time constructions for drawdown-constrained wealth processes can be found in \cite{cherny2013portfolio,kardaras2017numeraire}.

However, these continuous-time results do not transfer directly to discrete
time. 
In particular, \cite{klass2005grossman} show that the
Grossman--Zhou policy loses its optimality in discrete time, while
\cite{hernandez2023portfolio} characterize the optimal long-run growth
rate under a drawdown constraint through dynamic programming.
These constructions share the same high-level cushion principle, but
not the discrete-time robust control characterization underlying drawdown modulation.
In contrast to these asymptotic results, we consider a finite-horizon
setting. 
Our robust-safety analysis is distribution-free and guarantees the prescribed drawdown limit along every supported return path, and hence with probability
one.

Our approach connects closely to set-invariance methods for constrained control; see \cite{blanchini1999set,blanchini2015set}.
Recent approaches to robust constrained control include implicit invariant-set representations \cite{anevlavis2024controlled} and inner--outer set constructions \cite{comelli2024inner}.
Control barrier functions likewise enforce state constraints through restrictions on control inputs \cite{ames2017control,agrawal2017discrete}.
Here, we exploit the drawdown structure to characterize all robustly safe actions and, factorize the entire safe-action set by the current cushion.
This exact parameterization yields robust invariance to finite-horizon optimal policy synthesis.

\subsection{Contributions}

We establish a necessary and sufficient condition for robust drawdown invariance.
We prove an exact safe-action factorization that parameterizes every robustly safe control as the product of current cushion and a normalized safe action.

Under stagewise-independent returns, this
factorization reduces finite-horizon expected-wealth maximization over all robustly safe policies to a one-dimensional Bellman recursion.
We prove the optimal-value identity and the existence of a measurable optimal state-feedback~policy.

Under a prescribed drawdown limit, we prove that optimal drawdown modulation achieves no lower expected return than any feasible LTI policy.
An explicit sufficient condition for strict expected-return improvement
is also~established.

\emph{Notation.}
We write $\mathbb{R}_+:=[0,\infty)$. 
All vectors are regarded as column vectors, and $z^\top$ denotes the transpose of a vector $z$.
For $z\in\mathbb{R}^m$, the notation $|z|$ denotes the vector of
component-wise absolute values. 
Inequalities between vectors, as well as the operators $\min\{\cdot,\cdot\}$ and $\max\{\cdot,\cdot\}$ when applied to vectors, are understood component-wise. 
The zero and one vectors are denoted by $\mathbf{0}$ and $\mathbf{1}$, respectively, whenever its dimension is clear from context.
For any nonempty compact set $\mathcal{X}\subset\mathbb{R}^m$, define its support function~\cite{beck2017first} by
\[
\sigma_{\mathcal{X}}(z)
:=
\sup_{x\in\mathcal{X}}z^\top x.
\]

\section{Preliminaries}
\label{section: Preliminaries}
Fix a horizon $N \geq 1$.  For $k = 0,1,2,\dots,N$, we let $V(k)$ be the corresponding account value with $V(0) = V_0 >0$.

\begin{definition}[Percentage Drawdown]
For each stage $k = 0,1,2,\dots,N$, define the \emph{running maximum account value} by
$
V_{\max}(k) := \max_{0 \leq i \leq k} V(i).
$
The \emph{percentage drawdown at stage $k$} is 
\[
d(k) := \frac{V_{\max}(k) - V(k)}{V_{\max}(k)} 
\]
and the \emph{maximum percentage drawdown} over the horizon is given by
$$
d^\star  := \max_{0 \leq k \leq N}d(k)
$$
\end{definition}

Since $V_{\max}(k)\geq V_0>0$, these drawdown quantities are
well defined. 
Moreover, $V_{\max}(k)\geq V(k)$ implies that~$d(k)\geq0$ and hence $d^\star \geq0$. 
The upper bound~$d(k)\leq1$ requires $V(k)\geq0$.

\begin{remark}[Other Drawdown Metrics] 
	The analysis in this paper concerns the percentage drawdown, which measures loss relative to the running maximum account value. 
	An alternative is \emph{maximum absolute drawdown}, which measures the largest peak-to-trough loss in monetary units; see~\cite{Ismail_2004}.
\end{remark}

\subsection{Return Model and Support Information}

For each stage $k=0,1,\dots,N-1$, let
$$
X(k) := [ X_1(k) \; \cdots X_m(k) ]^\top \in \mathbb{R}^m
$$
denote the random vector of asset returns during stage~$k$, where $X_i(k)$ is the random return of asset $i$. 
For each $k$, assume that a known nonempty compact set~$\mathcal{X}_k\subset\mathbb{R}^m$ satisfies
$ 	
\mathbb{P}( X(k)\in\mathcal{X}_k) = 1.
$
No temporal independence assumption $\{X(k)\}_{k=0}^{N-1}$ is imposed at this stage; additional distributional assumptions will be introduced only for the
results that require them.

An important special case is the box-support model.
Given
$X_{\min}(k), X_{\max}(k) \in\mathbb{R}^m$ satisfying
$X_{\min}(k) \leq X_{\max}(k)$ componentwise, define
\begin{align*}
	\mathcal{X}_k
	&=
	[X_{\min}(k), \, X_{\max}(k)] \\
	&:=
	\{
	x\in\mathbb{R}^m:
	  x_i \in [X_{\min,i}(k), X_{\max,i}(k)],
	\ i=1,\dots,m
	\}.
\end{align*}
When the support set is time invariant, we write $\mathcal{X}_k = \mathcal{X}$ for all $k$.

\subsection{Account Dynamics and Admissible Policies}
Let $\mathcal{F}_0$ be the trivial $\sigma$-algebra and, for
$k=1,\dots,N$, define
$
\mathcal{F}_k
:=
\sigma\bigl(X(0),\dots,X(k-1)\bigr),
$
which represents the information available \emph{before} the stage-$k$ return is observed.
The control $u(k)\in\mathbb{R}^m$ is required to be $\mathcal{F}_k$-measurable; that is, the stage-$k$ investment may depend on past returns~$X(0), \dots, X(k-1)$ but not on $X(k)$.

Beginning from $V(0)=V_0>0$, the account value evolves according to
\[
V(k+1)
=
V(k)+u(k)^\top X(k)-\varepsilon^\top|u(k)|,
\]
where $\varepsilon\in\mathbb{R}_+^m$ is the vector of per-period proportional cost rates, so that $\varepsilon^\top|u(k)|$ represents a reduced-form proxy for execution costs, including bid--ask spreads and~fees.

\emph{Causal Policy.}
A \emph{causal policy} is a sequence~$\pi := \{ \pi_k\}_{k=0}^{N-1}$
of measurable decision rules~$
\pi_k : \mathcal{X}_0\times\cdots\times\mathcal{X}_{k-1} \to \mathbb{R}^m
$ 
that generate the actual control actions according to 
\begin{align*}
	&u^\pi(0):=\pi_0,\\
	\quad
	&u^\pi(k):=\pi_k\bigl(X(0),\dots,X(k-1)\bigr),
	\quad k=1,\dots,N-1.
\end{align*}

\begin{definition}[Admissible Causal Policy]
A causal policy $\pi= \{ \pi_k\}_{k=0}^{N-1}$
is \emph{admissible} if, for each $k=0,\dots,N-1$ and every state
reachable under $\pi$,
\[
V(k) + u^\pi(k)^\top x-\varepsilon^\top|u^\pi(k)|
\geq 0
\quad
\text{for all }x\in\mathcal{X}_k.
\]
Here, a reachable state is one generated from $V(0)=V_0$ under $\pi$
along some return history with $x(j)\in\mathcal{X}_j$.
We denote the class of all admissible causal policies by~$\Pi$.
\end{definition}

\subsection{Policy Performance}
For each $\pi \in \Pi$, let $V_\pi(k)$ denote the controlled account value by $\pi$, and let $d_\pi(k)$ denote the corresponding percentage drawdown.
Along any sample path~$\{V_\pi(k)\}_{k=0}^N$, we define the pathwise
overall return and maximum percentage drawdown by
\[
R_\pi
:=
\frac{V_\pi(N)-V_0}{V_0},
\qquad
d_\pi^\star 
:=
\max_{0\leq k\leq N}d_\pi(k),
\]
respectively.  
In what follows, we restrict attention to policies $\pi\in\Pi$ for which $R_\pi$ is integrable.

\section{Robust Drawdown-Safe Control}

\subsection{Robust Drawdown Invariance}

Fix a prescribed drawdown limit $d_{\max}\in[0,1)$. 
Using the augmented state $(V(k), V_{\max}(k))$, define the \emph{drawdown-safe set}
\begin{align} \label{eq: drawdown safe set}
	\mathcal{S}_{d_{\max}}
:=
\left\{
	(v,w)\in\mathbb{R}_+^2:
	w>0,\;
	(1-d_{\max})w\leq v\leq w
\right\}.
\end{align}
For each $k=0,1,\dots,N$, since $V_{\max}(k) >0$, the drawdown constraint $d(k) \leq d_{\max}$ is equivalent to 
\begin{align*} 
	(V(k), V_{\max}(k)) \in \mathcal{S}_{d_{\max}}.
\end{align*}
For any state in $\mathcal{S}_{d_{\max}}$, define the \emph{drawdown modulator}~by
\[
M(k)
:=
\frac{d_{\max}-d(k)}{1-d(k)}.
\]

\begin{definition}[Robust Drawdown Safety] \label{definition: robust drawdown safety}
We call a causal policy $\pi$ \emph{robustly drawdown-safe}
(or simply \emph{robustly safe}) if, for every stage
$k=0,\dots,N-1$ and every current state
$(V(k),V_{\max}(k))\in\mathcal{S}_{d_{\max}}$,
the control prescribed by $\pi$ yields
\[
(V(k+1),V_{\max}(k+1))\in\mathcal{S}_{d_{\max}}
\quad
\text{for every }x\in\mathcal{X}_k.
\]
In this case, we say that $\mathcal{S}_{d_{\max}}$ is
\emph{robustly invariant} under policy $\pi$.
\end{definition}

The following theorem gives an exact support-function characterization of robust invariance of
$\mathcal{S}_{d_{\max}}$.

\begin{theorem}[Robust Drawdown Invariance]
\label{thm:robust-drawdown-invariance}
Fix a prescribed drawdown limit $d_{\max}\in[0,1)$ and a stage $k \in \{0,\dots, N-1\}$. 
Suppose that $(V(k),V_{\max}(k))\in\mathcal{S}_{d_{\max}}$. 
For a given control action $u(k)$, the successor state $(V(k+1),V_{\max}(k+1))$ remains in~$\mathcal{S}_{d_{\max}}$ for every $x\in\mathcal{X}_k$ if and only if
\begin{align} \label{ineq:robust-drawdown-invariance}
	\sigma_{\mathcal{X}_k}\bigl(-u(k)\bigr)
+
\varepsilon^\top |u(k)|
\leq
M(k)V(k).
\end{align}
Consequently, a causal policy $\pi$ is robustly drawdown-safe if and only if ~\eqref{ineq:robust-drawdown-invariance} holds at every stage $k=0,\dots,N-1$ and every state in $\mathcal{S}_{d_{\max}}$. 
\end{theorem}

\begin{proof}
Fix a state
$
(v,w):=(V(k),V_{\max}(k))
\in\mathcal{S}_{d_{\max}}
$
and a control $u:=u(k)$. 
For a realization
$x\in\mathcal{X}_k$, define the successor state $(v^+, w^+)$ satisfying
\begin{align}
	v^+
	&:= v+u^\top x-\varepsilon^\top |u|, \label{eq: successor account value}
	\\
	w^+
	&:=\max\{w,v^+\}. \label{eq: maximum successor account}
\end{align}
The successor state $(v^+, w^+)$ belongs to $\mathcal{S}_{d_{\max}}$ if and only~if
$
w^+>0 \text{ and }  (1-d_{\max})w^+ \leq v^+ \leq w^+.
$
By the definition $w^+=\max\{w,v^+\}$ in \eqref{eq: maximum successor account}, we have
$w^+\geq w>0$ and $v^+\leq w^+$.
Hence,
$(v^+,w^+)\in\mathcal{S}_{d_{\max}}$ if and only if
\begin{align} \label{ineq: successor state drawdown condition}
	v^+
	\geq 
	(1-d_{\max})w^+.
\end{align}
We claim that this condition is equivalent to
\begin{align} \label{ineq: equivalent successor state drawdown condition}
v^+\geq(1-d_{\max})w.
\end{align}
To establish this, we observe that since $w^+ \geq w$ and $1-d_{\max} \geq 0$, Inequality \eqref{ineq: successor state drawdown condition} immediately implies
\eqref{ineq: equivalent successor state drawdown condition}.

Conversely, suppose that
\eqref{ineq: equivalent successor state drawdown condition} holds.
If $v^+\leq w$, by \eqref{eq: maximum successor account}, then $w^+=w$ and hence
\eqref{ineq: successor state drawdown condition} follows immediately.
On the other hand, if $v^+>w$, by \eqref{eq: maximum successor account} again, we have $w^+=v^+$.  
Since $1-d_{\max}\in (0,1]$, we have
\[
v^+ = w^+
\geq
(1-d_{\max})w^+.
\]
Hence, \eqref{ineq: successor state drawdown condition} holds. Therefore, \eqref{ineq: successor state drawdown condition} and
\eqref{ineq: equivalent successor state drawdown condition} are
equivalent.

Therefore, using \eqref{eq: successor account value} in~\eqref{ineq: equivalent successor state drawdown condition}, robust one-step invariance is equivalent to
\[
v+u^\top x-\varepsilon^\top |u|
\geq
(1-d_{\max})w
\quad
\text{for every }x\in\mathcal{X}_k.
\]
Rearranging and taking the worst case over $\mathcal{X}_k$ gives
\begin{align} \label{ineq: drawdown safety conditon}
\sup_{x\in\mathcal{X}_k}(-u^\top x)
+
\varepsilon^\top |u|
\leq
v-(1-d_{\max})w.
\end{align}
The left-hand side is
$
\sigma_{\mathcal{X}_k}(-u)+\varepsilon^\top |u|.
$

Finally, since
$
v=(1-d(k))w,
$
the right-hand side of~\eqref{ineq: drawdown safety conditon} satisfies
\begin{align*}
v-(1-d_{\max})w
&=
\bigl(d_{\max}-d(k)\bigr)w\\
&=
\frac{d_{\max}-d(k)}{1-d(k)}v
=
M(k)v.
\end{align*}
This proves the one-step equivalence. 
Applying the one-step equivalence at every stage and every
state in $\mathcal{S}_{d_{\max}}$ proves the necessary and
sufficient condition for robust invariance under $\pi$.
\end{proof}

\begin{corollary}[Box-Support Representation]
\label{cor:box-drawdown-invariance}
Suppose that the stage-$k$ support set is the box
$
\mathcal{X}_k =[X_{\min}(k),\, X_{\max}(k)] \subseteq \mathbb{R}^m.
$
Writing
\[
u^+(k):=\max\{u(k), \mathbf{0} \},
\quad
u^-(k):=\min\{u(k), \mathbf{0} \}
\]
componentwise, the robust drawdown-invariance condition~\eqref{ineq:robust-drawdown-invariance} becomes
\[
-X_{\min}(k)^\top u^+(k)
-
X_{\max}(k)^\top u^-(k)
+
\varepsilon^\top |u(k)|
\leq
M(k)V(k).
\]
\end{corollary}

\begin{proof}
For each component,
\begin{align*}
	&\sup_{x_i \in [X_{\min,i}(k), X_{\max,i}(k)]}
	\{-u_i(k)x_i\}\\
	 &\qquad =
	\begin{cases}
	-X_{\min,i}(k)u_i(k), & u_i(k)\geq0,\\[1ex]
	-X_{\max,i}(k)u_i(k), & u_i(k)<0.
	\end{cases}
\end{align*}
Summing over $i=1,\dots,m$ yields
\[
\sigma_{\mathcal{X}_k}\bigl(-u(k)\bigr)
=
-X_{\min}(k)^\top u^+(k)
-
X_{\max}(k)^\top u^-(k).
\]
The result follows from Theorem~\ref{thm:robust-drawdown-invariance}.
Under the cost-free, time-invariant box-support assumptions of
\cite{hsieh2017drawdown}, this resulting inequality reduces to
the condition in the Generalized Drawdown Modulation~Lemma.
\end{proof}

\begin{corollary}[Drawdown Guarantee and Survivability]
\label{cor:drawdown-survivability}
Suppose $X(k)\in\mathcal{X}_k$ with probability one for every $k=0,\dots,N-1$
and let $\pi = \{\pi_k\}_{k=0}^{N-1}$ be a robustly safe causal policy with induced account-value process $\{V(k)\}_{k=0}^{N}$.
Then
$
\mathbb{P}\bigl(d^\star\leq d_{\max}\bigr)=1.
$
Moreover, on the same probability-one event,
\[
V(k)
\geq
(1-d_{\max})V_{\max}(k)
>
0,
\qquad k=0,\dots,N.
\]
In particular,
$
\mathbb{P}(V(k) >0 \text{ for all } k=0,\dots,N) = 1. 
$
\end{corollary}

\begin{proof}
Define an event
$
\Omega_{\mathcal{X}}
:=
\bigcap_{k=0}^{N-1}
\left\{
X(k)\in\mathcal{X}_k
\right\}.
$
Note that
\begin{align*}
	\mathbb{P}(\Omega_{\mathcal{X}}^c) 
	&= \mathbb{P}\left( 
	\cup_{k=0}^{N-1} \{X(k) \notin \mathcal{X}_k\} 
	\right)\\
	&\leq \sum_{k=0}^{N-1} \mathbb{P}(X(k) \notin \mathcal{X}_k) = 0,
\end{align*}
where the last equality follows from the assumed hypothesis that $X(k)\in\mathcal{X}_k$ with probability one for every $k=0,\dots,N-1$; 
hence, we have~$\mathbb{P}(\Omega_{\mathcal{X}})=1$. 
On the event $\Omega_{\mathcal{X}}$, the assumed robust invariance of~$\mathcal{S}_{d_{\max}}$ under policy~$\pi$, together with~$(V_0,V_0)\in\mathcal{S}_{d_{\max}}$, implies
\[
\bigl(V(k),V_{\max}(k)\bigr)
\in
\mathcal{S}_{d_{\max}},
\quad k=0,\dots,N.
\]
The drawdown guarantee and wealth lower bound then follow directly
from the definition of $\mathcal{S}_{d_{\max}}$ and
$V_{\max}(k)\geq V_0$.
Indeed, for every $\omega\in\Omega_{\mathcal{X}}$ and every~$k=0,\dots,N$,
\[
V(k;\omega)
\geq
(1-d_{\max})V_{\max}(k;\omega)
\geq
(1-d_{\max})V_0
>
0.
\]
Hence,
$
\Omega_{\mathcal{X}}
\subseteq
\left\{
V(k)>0\text{ for all }k=0,\dots,N
\right\},
$
and the claim follows from
$\mathbb{P}(\Omega_{\mathcal{X}})=1$.
\end{proof}

\subsection{Exact Safe-Action Factorization}

The robust drawdown-invariance condition in Theorem~\ref{thm:robust-drawdown-invariance} admits an \emph{exact safe-action
factorization.} 
In particular, for each stage $k$, define the \emph{support-risk function} $h_k: \mathbb{R}^m \to \mathbb{R}$ by
\[
h_k(u)
:=
\sigma_{\mathcal{X}_k}(-u)
+
\varepsilon^\top|u|.
\]
Accordingly, Theorem~\ref{thm:robust-drawdown-invariance}
states that, whenever
$\bigl(V(k),V_{\max}(k)\bigr)\in\mathcal{S}_{d_{\max}}$,
a control action $u(k)$ keeps the successor state in
$\mathcal{S}_{d_{\max}}$ for every $x\in\mathcal{X}_k$ if and only if
$
h_k\bigl(u(k)\bigr)\leq C(k),
$
where 
\begin{align}\label{eq:cushion function}
	C(k)
	:=
	M(k)V(k) 
	= 
	V(k)-(1-d_{\max})V_{\max}(k)
\end{align}
is the \emph{drawdown cushion}.
We also define the \emph{normalized safe-action set}
\begin{align} \label{eq: normalized safe-action set}
	\Gamma_k
	:=
	\left\{
	\gamma\in\mathbb{R}^m:
	h_k(\gamma)\leq1
	\right\}.
\end{align}
For $c\geq0$, we write $c\Gamma_k:=\{c\gamma:\gamma\in\Gamma_k\}$
for the scalar multiple of $\Gamma_k$.

\begin{lemma}[Normalized Safe-Action Geometry]
\label{lemma:normalized-safe-action-geometry}
	For each stage $k$, the support-risk function $h_k$ is continuous and positively homogeneous. 
	If
	$
	h_k(u)>0
	$
	for every
	$
	u\in\mathbb{R}^m\setminus\{\mathbf{0}\},
	$
	then the normalized safe-action set $\Gamma_k$ is~compact.
\end{lemma}
\begin{proof}
Since $\mathcal{X}_k$ is compact, its support function $\sigma_{\mathcal{X}_k}$ is continuous; hence $h_k$ is also continuous.
Moreover, for every~$\alpha\geq0$, observe that
\begin{align*}
	h_k(\alpha u)
	&=
	\sigma_{\mathcal{X}_k}(-\alpha u)
	+
	\varepsilon^\top|\alpha u|\\
	&= \alpha (\sigma_{\mathcal{X}_k} (-u) + \varepsilon^\top |u|)
	=\alpha h_k(u),
\end{align*}
so $h_k$ is positively homogeneous.

Since $h_k$ is continuous and strictly positive on the compact unit sphere, the minimum of $h_k$ exists, and we define~$\eta_k:=\min_{u:\|u\|_2=1} h_k(u)>0$.
By positive homogeneity of $h_k$, for every $u\neq\mathbf{0}$,
\begin{align*}
	h_k(u) 
	&= h_k\left( \frac{u}{\|u\|_2} \|u\|_2 \right)\\
	&= \|u\|_2 \, h_k\left( \frac{u}{\|u\|_2} \right)
	\geq
	\eta_k\|u\|_2.
\end{align*}
This bound also holds at $u=\mathbf{0}$ because $h_k(\mathbf{0}) = 0$. 
Thus, we have
\[
h_k(u) \geq \eta_k \|u\|_2 \quad \text{for every } u \in \mathbb{R}^m.
\]
Therefore, every $\gamma \in \Gamma_k \subseteq \mathbb{R}^m$ satisfies $\eta_k \|\gamma\|_2 \leq h_k(\gamma) \leq 1 $, which implies $\|\gamma\|_2 \leq 1/\eta_k$; i.e., $\Gamma_k$ is bounded. 
It is also closed by continuity of $h_k$, and hence~$\Gamma_k$ is compact.
\end{proof}

\begin{remark}[Role of the Strict-Positivity Condition]
Since
\[
h_k(u)
=\sigma_{\mathcal{X}_k}(-u)
+
\varepsilon^\top|u|
=
-\min_{x\in\mathcal{X}_k}
\bigl(u^\top x-\varepsilon^\top|u|\bigr),
\]
the condition $h_k(u)>0$ for every $u\neq\mathbf{0}$ in
Lemma~\ref{lemma:normalized-safe-action-geometry}
means that every nonzero control action $u$ incurs a strictly negative net
payoff for at least one supported return.
In particular, it rules out a \emph{robust-arbitrage direction} whose net payoff is strictly positive for every supported return.
\end{remark}

\begin{theorem}[Exact Safe-Action Factorization]
\label{thm:exact-safe-action-factorization}
Fix a stage $k\in\{0,\dots,N-1\}$ and suppose that
$
h_k(u)>0
$
for every
$
u\in\mathbb{R}^m\setminus\{\mathbf{0}\}.
$
Whenever
$
\bigl(V(k),V_{\max}(k)\bigr)\in\mathcal{S}_{d_{\max}},
$
a control action $u(k)$ keeps the successor state
$
\bigl(V(k+1),V_{\max}(k+1)\bigr)
$
in $\mathcal{S}_{d_{\max}}$ for every $x\in\mathcal{X}_k$
if and only if
\[
u(k)\in C(k)\Gamma_k.
\]
Equivalently, there exists some $\gamma(k)\in\Gamma_k$ such that
\[
u(k)
=
C(k)\gamma(k)
=
M(k)V(k)\gamma(k).
\]
\end{theorem}
\begin{proof}
By Lemma~\ref{lemma:normalized-safe-action-geometry}, $h_k$ is
positively homogeneous. 
Suppose that
$\bigl(V(k),V_{\max}(k)\bigr)\in\mathcal{S}_{d_{\max}}$, which implies 
that $(1-d_{\max})V_{\max}(k) \leq V(k)$. 
By the definition of the cushion, $C(k)\geq0$.
Moreover,
according to Theorem~\ref{thm:robust-drawdown-invariance},
a control action $u(k)$ keeps the successor state in
$\mathcal{S}_{d_{\max}}$ for every $x\in\mathcal{X}_k$ if and only if
\begin{align} \label{ineq: h_k ineq}
	h_k\bigl(u(k)\bigr)\leq C(k).
\end{align}
We now consider two cases:

\emph{Case 1.} 
If $C(k)>0$, suppose that
$
h_k\bigl(u(k)\bigr)\leq C(k),
$
and define $\gamma(k):=u(k)/C(k)$. We first show $u(k)\in C(k)\Gamma_k.
$
Positive homogeneity of $h_k$ gives
\[
h_k(\gamma(k) ) 
= h_k \left(\frac{u(k)}{C(k)} \right) 
=
\frac{1}{C(k)}h_k(u(k))
\leq1,
\] where the last inequality holds by~\eqref{ineq: h_k ineq}.
Hence, $\gamma(k)\in\Gamma_k$ and $u(k)=C(k)\gamma(k)$. 
Conversely, if
$u(k)=C(k)\gamma(k)$ for some $\gamma(k)\in\Gamma_k$, then
\[
h_k\bigl(u(k)\bigr)
=
C(k)h_k(\gamma(k))
\leq C(k).
\]
Hence the stated factorization holds when $C(k)>0$.

\emph{Case 2.} If $C(k)=0$, the assumed strict positivity of $h_k$ away from
$u=\mathbf{0}$ implies
\begin{align} \label{eq:case2_for_modulated_policy}
\left\{
u\in\mathbb{R}^m:h_k(u)\leq0
\right\}
=
\{\mathbf{0}\}
=
0\Gamma_k. 
\end{align}
Therefore, $h_k\bigl(u(k)\bigr)\leq C(k)$ if and only if
$
u(k)\in C(k)\Gamma_k
$
also holds when $C(k)=0$.
Combining the two cases proves the result.
\end{proof}

\begin{remark}[Drawdown Modulation with Constant~$\gamma$]
We now specialize the safe-action factorization to drawdown-modulated
policies with constant $\gamma$, which is considered in \cite{hsieh2017drawdown,hsieh2023data}. 
Indeed, define
\begin{align} \label{eq: feasbile set for gamma}
	\Gamma
	:=
	\bigcap_{k=0}^{N-1}\Gamma_k
\end{align}
and fix any $\gamma\in\Gamma$.
Let $\pi^{\gamma,d_{\max}}$ denote the corresponding drawdown-modulated policy with constant $\gamma$, whose induced control satisfies
\begin{align} \label{eq: drawdown_modulated_polciy_with_constant_gamma}
u(k)
=
C(k)\gamma
\quad k=0,\dots,N-1,
\end{align}
where $C(k) = M(k) V(k)$ defined in \eqref{eq:cushion function}.
Whenever the current state belongs to $\mathcal{S}_{d_{\max}}$, we have~$C(k) \geq 0$. 
Since~$\gamma\in\Gamma\subseteq\Gamma_k$, \eqref{eq: normalized safe-action set} gives $h_k(\gamma)\leq1$. 
Together with~$C(k) \geq 0$ and positive homogeneity of~$h_k$, this~yields
$
h_k\bigl(u(k)\bigr)
=
C(k)h_k(\gamma)
\leq
C(k).
$
Hence, Theorem~\ref{thm:robust-drawdown-invariance} implies that the policy $\pi^{\gamma,d_{\max}}$ is robustly drawdown-safe.
\end{remark}

\section{Optimal Robust Drawdown-Safe Control}
\subsection{Robust Drawdown-Safe Control Problem}
To exclude nonzero actions with nonnegative net payoff for every
supported return in $\mathcal{X}_k$, we assume throughout this section  that
$
h_k(u)>0
$
for every  
$u\in\mathbb{R}^m\setminus\{\mathbf{0}\}$
and every 
$
k=0,\dots,N-1.
$
We now state the optimal robust drawdown-safe control problem:

\begin{problem}[Optimal Robust Drawdown-Safe Control]
\label{prob:finite-horizon-robust-safe-control}
For a fixed drawdown limit $d_{\max}\in[0,1)$, consider
\[
\begin{aligned}
\sup_{\pi\in\Pi}\quad
& \mathbb{E}[V_\pi(N)]\\
\text{s.t.}\quad
& \pi \text{ is robustly drawdown-safe},
\end{aligned}
\]
where robust drawdown safety is defined in
Definition~\ref{definition: robust drawdown safety}.
\end{problem}

\begin{remark}
	Since $V_0$ is fixed and $
	\mathbb{E}[R_\pi]
	=
	\frac{\mathbb{E}[V_\pi(N)]}{V_0}-1
	$, maximizing $\mathbb{E}[V_\pi(N)]$ is equivalent to maximizing $\mathbb{E}[R_\pi]$.
\end{remark}

The account dynamics and the robust-safety constraint are positively
homogeneous. Indeed, for any $\alpha>0$, if
$
\bigl(V(k),V_{\max}(k),u(k)\bigr)
$
is replaced by
$$
\bigl(\widetilde V(k),\widetilde V_{\max}(k),\widetilde u(k)\bigr)
 :=
\bigl(\alpha V(k), \alpha V_{\max}(k),\alpha u(k)\bigr),
$$
then, for the same return realization,
\[
\widetilde V(k+1)=\alpha V(k+1),
\quad
\widetilde V_{\max}(k+1)=\alpha V_{\max}(k+1).
\]
Consequently, the percentage drawdown is unchanged by this common
scaling.
We therefore separate the wealth scale from the account's relative
position within the drawdown-safe set.
Since $V_{\max}(k)\geq V_0>0$, define the \emph{normalized state}
\[
z(k)
:=
\frac{V(k)}{V_{\max}(k)}.
\]
For every policy $\pi$ feasible for Problem~\ref{prob:finite-horizon-robust-safe-control} and every return path in $\mathcal{X}_0\times\cdots\times\mathcal{X}_{N-1}$, we have
$
z(k)\in[1-d_{\max},1],
$
for $ k=0,\dots,N,
$
with
$
z(0)=1.
$
Moreover, the drawdown cushion~\eqref{eq:cushion function} satisfies
\[
C(k)
=
V_{\max}(k)
\left[
z(k)-(1-d_{\max})
\right].
\]
By Theorem~\ref{thm:exact-safe-action-factorization}, every robustly safe action $u(k)$ can be written as
\begin{align*}
	u(k) 
	= 
	C(k) \gamma(k)
	&=
	V_{\max}(k)
	\left[
	z(k)-(1-d_{\max})
	\right]
	\gamma(k),
\end{align*}
for some $\gamma(k)\in\Gamma_k$, where $\Gamma_k$ is the normalized safe-action set defined in~\eqref{eq: normalized safe-action set}.

\subsection{Reduced Bellman Recursion and Optimality}
Below, for the optimality analysis, we additionally assume that for $k=0,1,\dots, N-1$, return $X(k)$ is independent of $\mathcal{F}_k$ and has a stage-dependent distribution supported on $\mathcal{X}_k$.
For
$
z\in[1-d_{\max},1],
\gamma\in\Gamma_k,
$
and~$
x\in\mathcal{X}_k,
$
define the successor account value and the running maximum, both normalized by the current running maximum, as
\begin{align*}
	\widehat v^+(z,\gamma,x)
	&:=
	z+
	\left[
	z-(1-d_{\max})
	\right]
	\left(
	\gamma^\top x-\varepsilon^\top|\gamma|
	\right);\\
	\widehat w^+(z,\gamma,x)
	&:=
	\max\left\{
	1,\,
	\widehat v^+(z,\gamma,x)
	\right\},
\end{align*}
respectively. 
The successor normalized state is then
\[
z^+(z,\gamma,x)
:=
\frac{\widehat v^+(z,\gamma,x)}
     {\widehat w^+(z,\gamma,x)}.
\]
These quantities give the normalized state transition
\begin{align*}
V(k+1)
&=
V_{\max}(k)\,
\widehat v^+ \bigl(z(k),\gamma(k),X(k)\bigr),
\\
V_{\max}(k+1)
&=
V_{\max}(k)\,
\widehat w^+ \bigl(z(k),\gamma(k),X(k)\bigr),
\\
z(k+1)
&= \frac{V(k+1)}{V_{\max}(k+1)} 
=
z^+\bigl(z(k),\gamma(k),X(k)\bigr).
\end{align*}
Notably, robust safety guarantees
$
z^+(z,\gamma,x)
\in
[1-d_{\max},1].
$

\emph{Reduced Bellman Recursion.}
Fix $N\geq 1$.
For $z\in[1-d_{\max},1],$ set
$
J_N(z):=z,
$
and, for $k=N-1,\dots,0$, define~$J_k$ recursively by
\begin{align}
J_k(z)
:=
\sup_{\gamma\in\Gamma_k}
\mathbb{E}\Bigl[
\widehat w^+\bigl(z,\gamma,X(k)\bigr)
J_{k+1}\!\left(
z^+\bigl(z,\gamma,X(k)\bigr)
\right)
\Bigr].
\label{eq:reduced-bellman-recursion}
\end{align}
At this point, the functions $\{J_k\}_{k=0}^N$ are defined by the
backward recursion. Theorem~\ref{thm:reduced-bellman-recursion} below
shows that $wJ_k(v/w)$ equals the optimal expected terminal account
value from any stage-$k$ state $(v,w)$.

\begin{lemma}[Regularity and Pointwise Attainment]
\label{lem:bellman-regularity}
	For every $k=0,\dots,N$, the recursively defined function~$J_k$ is continuous on
	$[1-d_{\max},1]$. 
	Moreover, for every $k=0,\dots,N-1$ and every
	$z\in[1-d_{\max},1]$, the supremum in
	\eqref{eq:reduced-bellman-recursion} is attained.
\end{lemma}
\begin{proof}
We proceed by backward induction.
Note that the terminal function $J_N(z)=z$ is continuous. 
Suppose that $J_{k+1}$ is continuous, and define an auxiliary function
\[
Q_k(z,\gamma)
:=
\mathbb{E}\!\left[
\widehat w^+(z,\gamma, X(k))
J_{k+1}\!\left(
z^+\bigl(z,\gamma,X(k)\bigr)
\right)
\right].
\]
The integrand is continuous and bounded on the compact set
$
[1-d_{\max},1]\times\Gamma_k\times\mathcal{X}_k
$ 
where $\Gamma_k$ compactness follows from the assumption in Lemma~\ref{lemma:normalized-safe-action-geometry}.
Hence, by the dominated convergence theorem, $Q_k$ is
jointly continuous on
$
[1-d_{\max},1]\times\Gamma_k
$.
Since $\Gamma_k$ is nonempty, compact, and independent of $z$,
Berge's maximum theorem \cite[Theorem~17.31]{aliprantis2006infinite} implies~that
\[
J_k(z)=\max_{\gamma\in\Gamma_k}Q_k(z,\gamma)
\]
is continuous and that the maximum is attained for every~$z$.
Backward induction completes the proof.
\end{proof}

For $k=0,\dots,N$, let $\Pi_k^{\mathrm{safe}}$ denote the robustly
safe causal policy sequences over stages $k,\dots,N-1$, and let~$\mathbb{E}_k^{v,w}[\cdot]$ denote expectation for the process initialized at
stage $k$ from
$
(V(k),V_{\max}(k))=(v,w).
$
We have the following optimality result for the Bellman recursion.

\begin{theorem}[Optimality of the Bellman Recursion]
\label{thm:reduced-bellman-recursion}
For every $k=0,\dots,N$ and
$(v,w)\in\mathcal{S}_{d_{\max}}$, we have the optimal-value identity
\begin{align}\label{eq: optimal-value identity}
	\sup_{\pi\in\Pi_k^{\mathrm{safe}}}
	\mathbb{E}_k^{v,w}\!\left[V_\pi(N)\right]
	=
	wJ_k\!\left(\frac{v}{w}\right).
\end{align}
Moreover, for each $k=0,\dots, N-1$, there exists a Borel-measurable function
$
\gamma_k^\star: [1-d_{\max}, 1] \to \Gamma_k
$ 
that attains the maximum in \eqref{eq:reduced-bellman-recursion} for every
$z\in[1-d_{\max},1]$. 
The resulting feedback law
\[
u^{\pi^\star}(k)
=
V_{\max}(k)
\left[z(k)-(1-d_{\max})\right]
\gamma_k^\star\bigl(z(k)\bigr)
\]
defines a causal policy $\pi^\star$ that renders
$\mathcal{S}_{d_{\max}}$ robustly invariant and attains the supremum in the optimal-value identity. In particular,
\[
\mathbb{E}[V_{\pi^\star}(N)]
=
V_0J_0(1),
\qquad
\mathbb{E}[R_{\pi^\star}]
=
J_0(1)-1.
\]
\end{theorem}
\begin{proof} 
Under the standing assumption that return $X(k)$ is stage-wise independent of $\mathcal{F}_k$ for every $k=0,\dots, N-1$, the continuation problem at stage $k$ depends on past returns only through the current state $(v,w)$. 
We prove the optimal-value identity~\eqref{eq: optimal-value identity} by backward induction.

Indeed, at terminal stage $k=N$, no decisions remain, and the terminal account value is
\[
v
=
w\frac{v}{w}
=
wJ_N\!\left(\frac{v}{w}\right).
\]
Thus, the assertion holds at the terminal stage.

Now suppose that the optimal-value identity~\eqref{eq: optimal-value identity} holds at stage $k+1$, and consider a stage-$k$ state $(v,w)$. Write
$
z=v/w.
$
By exact safe-action factorization Theorem~\ref{thm:exact-safe-action-factorization}, every robustly safe action $u$ can be factorized~as
\[
u
= C\gamma
\; \text{ with } \;
\gamma\in\Gamma_k,
\]
where $C = w\left[z-(1-d_{\max})\right]$.
If the stage-$k$ return is~$x\in\mathcal{X}_k$, the successor account value and running maximum are
\[
v^+
=
w\,\widehat v^+(z,\gamma,x),
\qquad
w^+
=
w\,\widehat w^+(z,\gamma,x),
\]
and hence
$
\frac{v^+}{w^+}
=
z^+(z,\gamma,x).
$
By the induction hypothesis, we have
\begin{align*}
\sup_{\pi\in\Pi_{k+1}^{\mathrm{safe}}}
\mathbb{E}_{k+1}^{v^+,w^+}\!\left[V_\pi(N)\right]
&=
w^+J_{k+1}\!\left(\frac{v^+}{w^+}\right)\\
&=
w\,\widehat w^+(z,\gamma,x)
J_{k+1}\!\left(z^+(z,\gamma,x)\right).
\end{align*}
For the chosen $\gamma\in\Gamma_k$, the preceding identity holds
for every $x\in\mathcal{X}_k$.
Evaluating the right-hand side at~$x=X(k)$ and taking expectation gives the expected terminal account value under optimal subsequent controls; i.e.,~$
w\,
\mathbb{E}\!\left[
\widehat w^+\bigl(z,\gamma,X(k)\bigr)
J_{k+1}\!\left(
z^+\bigl(z,\gamma,X(k)\bigr)
\right)
\right].
$
Taking the supremum over $\gamma\in\Gamma_k$ gives
\begin{align*}
&w
\sup_{\gamma\in\Gamma_k}
\mathbb{E}\!\left[
\widehat w^+\bigl(z,\gamma,X(k)\bigr)
J_{k+1}\!\left(
z^+\bigl(z,\gamma,X(k)\bigr)
\right)
\right]
\\
&\qquad
=
wJ_k(z)
=
wJ_k\!\left(\frac{v}{w}\right),
\end{align*}
where the first equality follows from
\eqref{eq:reduced-bellman-recursion}.
Thus, the assertion holds at stage $k$, and backward induction proves the claimed optimal-value representation.

It remains to show that this supremum is attained by a causal robustly
safe policy.
The proof of Lemma~\ref{lem:bellman-regularity} shows that the Bellman
objective in \eqref{eq:reduced-bellman-recursion} is continuous in $(z,\gamma)$.
Since $\Gamma_k$ is compact, the measurable maximum theorem \cite[Theorem~18.19]{aliprantis2006infinite} provides,
for every $k=0,\dots,N-1$, a Borel-measurable maximizing selector
$
\gamma_k^\star:
[1-d_{\max},1]
\to
\Gamma_k.
$
Since both~$C(k)$ and $z(k)$ are $\mathcal{F}_k$-measurable, the
feedback law~$
u^{\pi^\star}(k)
=
C(k)\gamma_k^\star(z(k))
$
is also $\mathcal{F}_k$-measurable.
Thus, the maximizing selectors $\gamma_k^\star$ define a causal
policy~$\pi^\star$.

Moreover, since
$
\gamma_k^\star(z(k))\in\Gamma_k
$
and
$
C(k)
=
V_{\max}(k)[z(k)-(1-d_{\max})],
$
positive homogeneity gives
\[
h_k\bigl(u^{\pi^\star}(k)\bigr)
=
C(k)
h_k\!\left(\gamma_k^\star(z(k))\right)
\leq
C(k).
\]
Theorem~\ref{thm:robust-drawdown-invariance} therefore implies that
$\pi^\star$ renders
$\mathcal{S}_{d_{\max}}$ robustly invariant.

Because $\gamma_k^\star$ attains the Bellman maximum at every stage
and state, backward induction yields, for every $k=0,\dots,N$ and
$(v,w)\in\mathcal{S}_{d_{\max}}$,
\begin{align*}
\mathbb{E}_k^{v,w}\!\left[V_{\pi^\star}(N)\right]
&=
wJ_k\!\left(\frac{v}{w}\right)
\\
&=
\sup_{\pi\in\Pi_k^{\mathrm{safe}}}
\mathbb{E}_k^{v,w}\!\left[V_\pi(N)\right].
\end{align*}
At the initial state,
$
(V(0),V_{\max}(0))
=
(V_0,V_0)
$
and
$
z(0)=1,
$
and consequently
$
\mathbb{E}[V_{\pi^\star}(N)]
=
V_0J_0(1).
$
Thus,~$\pi^\star$ is globally optimal for
Problem~\ref{prob:finite-horizon-robust-safe-control}.
Finally,
\[
\mathbb{E}[R_{\pi^\star}]
=
\frac{\mathbb{E}[V_{\pi^\star}(N)]-V_0}{V_0}
=
J_0(1)-1. \qedhere
\]
\end{proof}

\section{Comparison with LTI Policies}
\label{section:comparison with LTI policies}
This section compares the drawdown-modulated and LTI policies under a prescribed drawdown limit.
We first characterize the LTI gains satisfying this limit over the horizon, then establish conditions for strict improvement by the optimal drawdown-modulated policy.

\subsection{Drawdown Guarantee for LTI Benchmark} 
\label{subsection: LTI feedback benchmark}
As a benchmark, we consider constant-gain linear time-invariant (LTI) feedback policies within the admissible causal policy class
$\Pi$. 
Throughout this section, we assume that the return-support information is time-invariant
$
X(k)\in\mathcal{X}
$
almost surely, where $\mathcal{X}\subset\mathbb{R}^m$ is nonempty and compact.
We further assume the following conditional-mean condition: for stage $k=0,\dots,N-1,$
\begin{align}\label{eq: conditional-mean assumption} 
	\mathbb{E}[X(k) \mid \mathcal{F}_k]
	=
	\mu
	\quad \text{almost surely},  
\end{align}
for some $\mu\in\mathbb{R}^m$. 

For each feedback gain vector $K=[K_1\;\cdots\;K_m]^\top\in\mathbb{R}^m$, let $\pi^K := \{\pi_k^K\}_{k=0}^{N-1}$ denote the corresponding policy, whose induced control satisfies
\[
	u_K(k) 
	= \pi_k^K \bigl(X(0), \dots, X(k-1) \bigr)
	:= K V_K(k).
\]
The corresponding admissible feedback gain set is
\begin{align} \label{eq: admissible feedback gain set}
	\mathcal{K}
	:=
	\left\{
	K\in\mathbb{R}^m:
	\sigma_{\mathcal{X}}(-K) + \varepsilon^\top |K| \leq1
	\right\}.
\end{align}
This set is closed and convex, irrespective of whether~$\mathcal{X}$ is convex.\footnote{ 
Indeed, the support function $\sigma_{\mathcal{X}}$ is convex and lower semicontinuous, and hence so is the mapping~$K \mapsto \sigma_{\mathcal{X}}(-K) + \varepsilon^\top |K|$.
Therefore, $\mathcal{K}$ is a closed convex sublevel set; see \cite{beck2017first}.
}
Hence, every $K\in\mathcal{K}$ induces an admissible causal policy, and
the LTI benchmark class is
\[
\Pi_{\mathrm{LTI}}
:=
\left\{
\pi^K :K\in\mathcal{K}
\right\}
\subseteq\Pi.
\]
Beginning from $V_K(0)=V_0>0$, the LTI account value satisfies
\begin{align}
	V_K(k+1)
	&=
	V_K(k)+K^\top X(k)V_K(k) -\varepsilon^\top|KV_K(k)| \notag \\ 
	&=
	\bigl(1+K^\top X(k)-\varepsilon^\top|K|\bigr)V_K(k). \label{eq: LTI dynamics}
\end{align}
For each $K\in\mathcal{K}$, write
$
R_K:=R_{\pi^K},
$
and
$
d_K^\star:=d_{\pi^K}^\star.
$
Since $V_K(k)$ is $\mathcal{F}_k$-measurable, the conditional-mean
assumption gives
\begin{align*}
	&\mathbb{E}[V_K(k+1)\mid\mathcal{F}_k]\\
	&=
	V_K(k)
	\left(
	1+K^\top
	\mathbb{E}[X(k)\mid\mathcal{F}_k] - \varepsilon^\top |K|
	\right)\\
	&=
	(1+K^\top\mu - \varepsilon^\top |K|)V_K(k).
\end{align*}
Taking expectations and iterating from $V_K(0)=V_0$ yields
$
\mathbb{E}[V_K(N)]
=
V_0(1+K^\top\mu -\varepsilon^\top |K|)^N.
$
Therefore,
\[
\mathbb{E}[R_K]
=
(1+K^\top\mu - \varepsilon^\top |K|)^N-1.
\]
To compare policies under a common prescribed drawdown limit,
we next characterize the LTI gains that satisfy this limit
over the horizon.  

\begin{proposition}[LTI Drawdown Guarantee]
\label{prop:lti-finite-horizon-drawdown}
	Fix $d_{\max}\in[0,1)$ and $K\in\mathcal{K}$.
	Starting from~$V_K(0)=V_0>0$, the LTI policy $\pi^K$ satisfies~$d_K^\star\leq d_{\max}$ for every return path in $\mathcal{X}^N$ if and only if
	\begin{align}
	h_k(K)
	\leq
	1-(1-d_{\max})^{1/N}.
	\label{ineq:lti-finite-horizon-drawdown}
	\end{align}
where $h_k(K)=\sigma_{\mathcal{X}}(-K)+\varepsilon^\top|K|$ is independent of $k$.
\end{proposition}

\begin{proof}
Let $K\in \mathcal{K}$ and note that
\begin{align*}
\inf_{x\in\mathcal{X}}
\left(1+K^\top x-\varepsilon^\top|K|\right)
&=
1-\sigma_{\mathcal{X}}(-K)-\varepsilon^\top|K|
\geq0,
\end{align*}
where nonnegativity follows from $K\in\mathcal{K}$ defined in \eqref{eq: admissible feedback gain set}.

For necessity, compactness of $\mathcal{X}$ implies that
the minimum is attained at some $x^\star\in\mathcal{X}$.
Along the constant return path $x(k)=x^\star$,
$k=0,\dots,N-1$, the account value satisfies
\begin{align} \label{eq: LTI cumulative account}
	V_K(N)=(1-\sigma_{\mathcal{X}}(-K)-\varepsilon^\top|K|)^N V_0.
\end{align}
Using the prescribed drawdown bound $d_K^\star \leq d_{\max}$ implies $d(k) \leq d_{\max}$ for all $k$; i.e., $V(k) \geq (1-d_{\max})V_{\max}(k)$ for all $k$, which also holds at $k=N$.
Hence, applying \eqref{eq: LTI cumulative account} on the left-hand side gives
\begin{align*}
	(1-\sigma_{\mathcal{X}}(-K)-\varepsilon^\top|K|)^N V_0 
	&\geq (1-d_{\max})V_{\max}(N)\\
	&\geq (1-d_{\max})V_0,
\end{align*}
where the last inequality holds by the fact that the running maximum is at least $V_0$. Cancelling out $V_0>0$ on both sides and taking the $N$th root gives
$$
(1-\sigma_{\mathcal{X}}(-K)-\varepsilon^\top|K|)\geq(1-d_{\max})^{1/N},
$$ 
which is equivalent to \eqref{ineq:lti-finite-horizon-drawdown}.

Conversely, suppose \eqref{ineq:lti-finite-horizon-drawdown} holds.
The LTI account dynamics~\eqref{eq: LTI dynamics}~gives
\begin{align*}
	V_K(k+1)
	&=
	\bigl(1+K^\top X(k)-\varepsilon^\top|K|\bigr)V_K(k)\\
	&\geq \bigl( 1-\sigma_{\mathcal{X}}(-K) - \varepsilon^\top|K| \bigr) V_K(k)\\
	& \geq ( 1-d_{\max})^{1/N}  V_K(k),
\end{align*}
where the last inequality uses \eqref{ineq:lti-finite-horizon-drawdown}. 
Thus, the one-stage relative loss cannot exceed $1-(1-d_{\max})^{1/N}$.

Fix any return path in $\mathcal{X}^N$ and any
$k\in\{0,\dots,N\}$.
Choose $i^* \in\{0,\dots,k\}$ such that
$
V_K(i^*)= V_{\max}(k).
$
Iterating the account dynamics from $i^*$ to $k$ gives
\begin{align}
V_K(k)
& \geq
(1-d_{\max})^{(k-i^*)/N}V_K(i^*) \notag \\
& \geq
(1-d_{\max})V_K(i^*)	\label{ineq:intermediate ineq}\\
&= 
(1-d_{\max})V_{\max}(k), \label{ineq:drawdown ineq in Vmax form}
\end{align}
where the last inequality~\eqref{ineq:intermediate ineq} uses $0\leq \frac{k-i^*}{N} \leq 1$.
Consequently, \eqref{ineq:drawdown ineq in Vmax form} implies that the drawdown at stage $k$ satisfies
\[
d_K(k)
=
1-\frac{V_K(k)}{V_{\max}(k)}
\leq d_{\max}.
\]
Since the return path and stage $k$ were arbitrary,
$d_K^\star\leq d_{\max}$ along every return path in
$\mathcal{X}^N$.
\end{proof}

For a prescribed drawdown limit $d_{\max}\in[0,1)$,
Proposition~\ref{prop:lti-finite-horizon-drawdown}
gives the following LTI benchmark optimization problem:
\begin{align} \label{problem: LTI optimization problem}
	\max_{K\in\mathcal{K}}
	\left\{
	\mathbb{E}[R_K]:
	K \text{ satisfies }
	\eqref{ineq:lti-finite-horizon-drawdown}
	\right\},
\end{align}
where
$\mathbb{E}[R_K]
=(1+K^\top\mu-\varepsilon^\top|K|)^N-1$.

\begin{remark}[LTI Class as a Special Case of Drawdown Modulation]
Under the time-invariant support assumption, $\Gamma=\mathcal{K}$.
Extending the modulator by the convention $M(k)=1$ when $d_{\max}=1$, the drawdown-modulated law reduces to
$
u(k)=\gamma V(k).
$
Thus every admissible LTI policy is recovered by setting $\gamma:=K$; see also \cite{hsieh2017drawdown,hsieh2017inefficiency}.
\end{remark}

\subsection{Outperformance of Drawdown Modulation}
Under a prescribed drawdown limit $d_{\max}$, the following theorem indicates that optimal drawdown-modulated policy achieves no lower expected return
than any LTI policy satisfying that limit.

\begin{theorem}[Weak and Strict Improvement over LTI Policies]
\label{thm:exact-strict-domination}
	Assume $N\geq 2$ and that, for each $k=0,\dots,N-1$,
	$X(k)$ is independent of $\mathcal{F}_k$, and $h_k(u)>0$ for every $u\neq\mathbf{0}$.
	Fix~$d_{\max}\in[0,1)$ and $K\in\mathcal{K}$ satisfying~\eqref{ineq:lti-finite-horizon-drawdown}.
	For this drawdown limit, let $J_0$ denote the stage-$0$ Bellman value function and let $\pi^\star$ be an optimal drawdown-modulated policy provided by Theorem~\ref{thm:reduced-bellman-recursion}.
	Then, we have
	\begin{align} \label{ineq:weak outperformance}
	\mathbb{E}[R_{\pi^\star}]
	=
	J_0(1)-1
	\geq
	\mathbb{E}[R_K].
	\end{align}
	The inequality is strict whenever 
	$
	K^\top\mu-\varepsilon^\top|K|>0.
	$
\end{theorem}

\begin{proof}
Fix $d_{\max}\in[0,1)$ and $K\in\mathcal{K}$ satisfying \eqref{ineq:lti-finite-horizon-drawdown}.
Proposition~\ref{prop:lti-finite-horizon-drawdown} indicates that the LTI policy $\pi^K$ satisfies the prescribed drawdown limit
along every return path in~$\mathcal{X}^N$, which implies that the LTI control action~$u_K(k)=KV_K(k)$ is robustly safe at every state reachable under $\pi^K$ at each stage $k=0,\dots,N-1$.

We now use the drawdown-modulation characterization to construct a robustly safe policy with the same performance as the LTI policy above.
Specifically, for any current state~$(V(k), V_{\max}(k)) \in \mathcal{S}_{d_{\max}}$, consider the control action
\[
\widetilde u(k)
:=
M(k)V(k)\widetilde\gamma(k),
\]
where
\[
\widetilde\gamma(k)
:=
\begin{cases}
	K/M(k),
	& M(k)>0 \text{ and } K/M(k)\in\Gamma,\\
	\mathbf{0},
	& \text{otherwise}.
\end{cases}
\]
This state-feedback rule defines a causal policy $\widetilde\pi$.
Since the support set $\mathcal{X}$ is time invariant, it follows that
$\Gamma_k=\Gamma$ for every $k$.
Thus $\widetilde\gamma(k)\in\Gamma_k$, and Theorem~\ref{thm:exact-safe-action-factorization} implies that $\widetilde\pi$ renders~$\mathcal{S}_{d_{\max}}$ robustly invariant.

Now consider any state reachable under $\pi^K$ at a stage~$k<N$.
If $M(k)>0$, the Factorization Theorem~\ref{thm:exact-safe-action-factorization} therefore gives $K/M(k)\in\Gamma$.
If $M(k)=0$, the same theorem implies $KV_K(k)=\mathbf{0}$.
Hence $\widetilde\pi$ reproduces the LTI action at every
such state.
Starting from the same initial account value, induction yields
\[
	V_{\tilde{\pi}}(k)=V_K(k),
	\qquad k=0,\dots,N,
\]
along every return path in $\mathcal{X}^N$.
Bellman optimality in Theorem~\ref{thm:reduced-bellman-recursion} therefore gives
\[
	J_0(1)-1
	=
	\mathbb{E}[R_{\pi^\star}]
	\geq
	\mathbb{E}[R_{\tilde\pi}]
	=
	\mathbb{E}[R_K].
\]

To prove strictness of~\eqref{ineq:weak outperformance}, suppose additionally that $K^\top\mu-\varepsilon^\top|K|>0$, which implies $K\neq\mathbf{0}$.
Since the support set is time invariant, $h_k(K)$
is independent of $k$.
The assumption $h_k(K)>0$ and
\eqref{ineq:lti-finite-horizon-drawdown} imply
\[
0<h_k(K)<1,
\qquad
\bigl(1-h_k(K)\bigr)^N\geq1-d_{\max}.
\]
Define a policy $\bar\pi$ to coincide with $\widetilde\pi$ through stage $N-2$, and set its control action at stage $N-1$ as
\[
u^{\bar\pi}(N-1)
:=
\frac{K}{h_k(K)}C(N-1),
\]
where $C(N-1)=M(N-1)V(N-1)$ is the cushion under $\bar\pi$.
Since $h_k(K) >0$ and $h_k(\cdot)$ is positively homogeneous, we have $h_k\!\left(\frac{K}{h_k(K)}\right)=1$.
Thus, $K/h_k(K)\in\Gamma$, and the modified policy $\bar\pi$ is robustly safe.

Write
$
v:=V_K(N-1),
$
and
$
w:=\max_{0\leq j\leq N-1}V_K(j).
$
By Theorem~\ref{thm:robust-drawdown-invariance},
the robust safety of the LTI action $u=Kv$ at $(v,w)$ gives
\begin{align} \label{ineq: LTI saftey}
h_k(Kv)
\leq
C =
v-(1-d_{\max})w. 
\end{align}
Since $h_k(Kv)=v h_k(K)$ by positive homogeneity, \eqref{ineq: LTI saftey} becomes
\[
\bigl(1-h_k(K)\bigr)v
\geq
(1-d_{\max})w.
\]
Moreover, by assumed condition $K^\top\mu-\varepsilon^\top|K|>0$, we have
\begin{align*}
	\mathbb{E}[V_K(1)]
	&=
	V_0(1+K^\top\mu-\varepsilon^\top|K|)
	>V_0,
\end{align*}
so
$\mathbb{P}(V_K(1)>V_0)>0$.
On this event, a time attaining $w$ lies in
$\{1,\dots,N-1\}$.
Since every LTI wealth multiplier is at least $1-h_k(K)$,
\[
v \geq \bigl(1-h_k(K)\bigr)^{N-2}w,
\]
and
\[
\bigl(1-h_k(K)\bigr)v
\geq
\bigl(1-h_k(K)\bigr)^{N-1}w
>
(1-d_{\max})w.
\]
The policies $\bar\pi$ and $\pi^K$ have the same wealth
history through stage $N-1$. Hence,
\[
V_{\bar\pi}(N-1)=V_K(N-1)=v,
\quad
C(N-1)=v-(1-d_{\max})w.
\]
Since $v$ and $C(N-1)$ are nonnegative and
$\mathcal{F}_{N-1}$-measurable, the account dynamics and conditional mean assumption~\eqref{eq: conditional-mean assumption} give
\begin{align*}
\mathbb{E}\!\left[V_{\bar\pi}(N)\mid\mathcal{F}_{N-1}\right]
&=
v+\frac{C(N-1)}{h_k(K)}
\left(K^\top\mu-\varepsilon^\top|K|\right),\\
\mathbb{E}\!\left[V_K(N)\mid\mathcal{F}_{N-1}\right]
&=
v+v\left(K^\top\mu-\varepsilon^\top|K|\right).
\end{align*}
Subtracting these identities and invoking the tower property of expectations yields
\begin{align*}
&\mathbb{E}[V_{\bar\pi}(N)-V_K(N)]\\
&=
\left(K^\top\mu-\varepsilon^\top|K|\right)
\mathbb{E}\!\left[\frac{C(N-1)}{h_k(K)}-v\right]\\
&=
\left(K^\top\mu-\varepsilon^\top|K|\right)
\mathbb{E}\!\left[\frac{v - (1-d_{\max})w}{h_k(K)}-v\right]\\
&=
\frac{K^\top\mu-\varepsilon^\top|K|}{h_k(K)}
\mathbb{E}\!\left[
\bigl(1-h_k(K)\bigr)v-(1-d_{\max})w
\right]\\
&>0.
\end{align*}
The last inequality follows because the prefactor is
positive, while the bracketed term is nonnegative on every
supported path and, as shown above, strictly positive on
the event $\{V_K(1)>V_0\}$, which has positive probability.
Optimality of $\pi^\star$ proves the strict inequality.
\end{proof}

\section{Numerical Illustration}
\label{sec:numerical-illustration}
This section provides an numerical illustrations for the LTI benchmark with a drawdown-modulated policy computed from the reduced Bellman recursion under a common prescribed drawdown limit $d_{\max}$.

\subsection{Calibration and Setup}

We begin by constructing a joint empirical return distribution from daily adjusted closing prices of \texttt{SPY} and \texttt{TLT} obtained from Yahoo Finance, from January 2015 through July~2026. 
Each of the $2{,}910$ observed daily joint return vectors is assigned equal probability, and their set defines the support $\mathcal{X}$.
Simulation paths are generated by sampling joint return vectors independently across stages from this empirical distribution.

We take $V_0=1$, horizons of $N\in\{21,63,252\}$ trading days, and cost rates
$\varepsilon=5\times10^{-5}\mathbf{1}$, corresponding to $0.5$ basis
points per asset and stage.
Note that the origin lies in the interior of the convex hull of $\mathcal{X}$, implying $h_k(u)>0$ for every $u\neq\mathbf{0}$. 
Hence, the normalized safe-action set $\Gamma_k$ is compact by Lemma~\ref{lemma:normalized-safe-action-geometry}.
For each horizon and each prescribed drawdown limit~$d_{\max}\in\{0,0.05,\dots,0.60\}$, we compute both LTI and drawdown-modulated policies.

\emph{LTI Benchmark.}
The LTI feedback gain $K$ solves~\eqref{problem: LTI optimization problem}, which is equivalent to maximize $K^\top\mu-\varepsilon^\top|K|$ subject to~\eqref{ineq:lti-finite-horizon-drawdown}, where~$\mu$ is the empirical mean return. 
This can be solved by linear programming, and its expected terminal return is computed as~$
\mathbb{E}[R_K]
=
\bigl(1+K^\top\mu-\varepsilon^\top|K|\bigr)^N-1.
$

\emph{Drawdown-Modulated Policy.}
For $d_{\max}>0$, we numerically approximate the Bellman recursion~\eqref{eq:reduced-bellman-recursion}, using feasible normalized actions in $\Gamma_k$.
The expected return is estimated from $300{,}000$ independently simulated paths of length $N$. 
At $d_{\max}=0$, both policies prescribe $u(k)=\mathbf{0}$, so $V(k)=V_0$ throughout the horizon.

\subsection{Performance under a Pathwise Drawdown Limit}

Figure~\ref{fig:numerical-fixed-cap-comparison} shows higher estimated
expected returns for the proposed drawdown-modulated policy than for
the optimized LTI benchmark at every positive evaluated drawdown limit.
For $N=21$ and $N=63$, the numerically computed optimal drawdown modulated policy uses an approximately constant $\gamma$ across stages and states; 
the observed improvement in these panels is therefore already attained by drawdown modulation with constant~$\gamma$.

\begin{figure*}[htbp]
	\centering
	\includegraphics[width=0.95\textwidth]
	{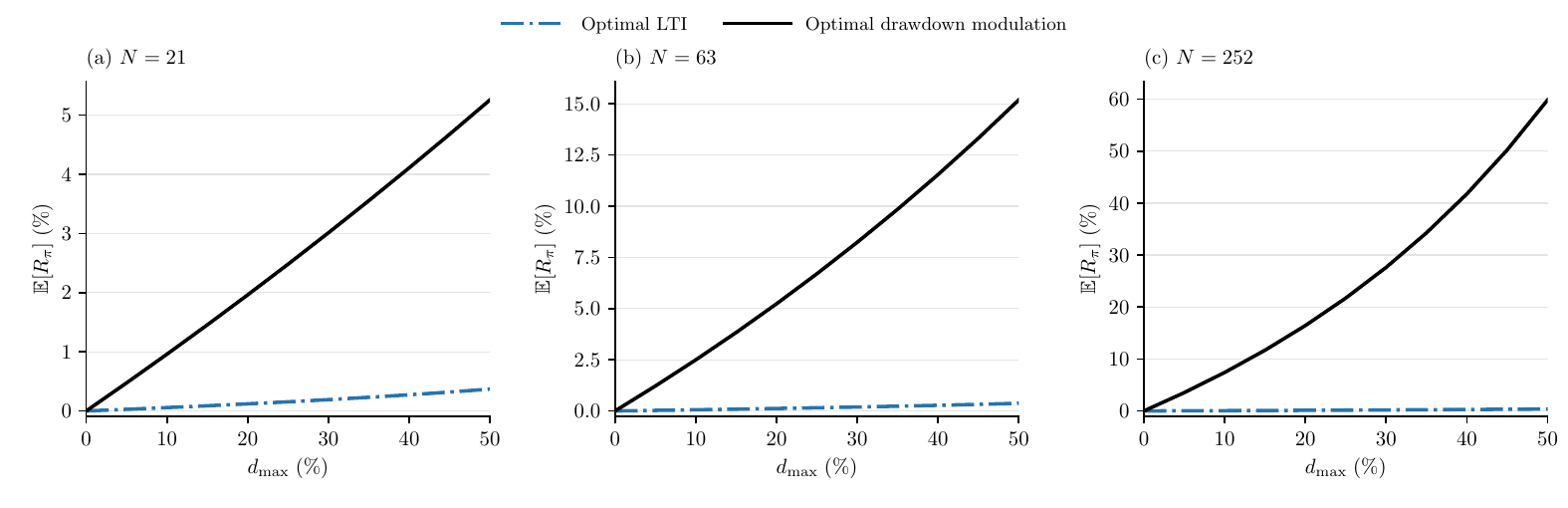}
	\caption{Expected terminal return under a prescribed drawdown
	limit $d_{\max}$ for (a) $N=21$, (b) $N=63$, and (c) $N=252$ trading days.
	}
	\label{fig:numerical-fixed-cap-comparison}
\end{figure*}

\section{Conclusion}
This paper developed a finite-horizon robust-control framework for maximum percentage drawdown in discrete-time stochastic systems. 
By expressing the drawdown constraint as robust invariance of an augmented wealth state, we obtained an exact support-function characterization of robustly safe control actions. 
The safe-action set admits an exact cushion-based factorization that parameterizes all robustly safe controls. 
Under stagewise-independent returns, homogeneity reduces optimal robust drawdown-safe control to a one-dimensional Bellman recursion and yields a globally optimal state-feedback policy. 
Comparison with the LTI benchmark further shows that optimal drawdown modulated policy achieves no lower expected return than any LTI policy satisfying the same prescribed drawdown limit, with strict improvement under explicit sufficient~conditions.


\bibliographystyle{apalike}    
\bibliography{refs}           

@inproceedings{agrawal2017discrete,
  author    = {Agrawal, Ayush and Sreenath, Koushil},
  title     = {{Discrete Control Barrier Functions for
               Safety-Critical Control of Discrete Systems
               with Application to Bipedal Robot Navigation}},
  booktitle = {Proceedings of Robotics: Science and Systems},
  year      = {2017}
}

@book{aliprantis2006infinite,
  title     = {{Infinite Dimensional Analysis: A Hitchhiker's Guide}},
  author    = {Aliprantis, Charalambos D. and Border, Kim C.},
  edition   = {3rd},
  year      = {2006},
  publisher = {Springer-Verlag}
}

@inproceedings{barmish2024jump,
	title={{A Jump Start to Stock Trading Research for the Uninitiated Control Scientist: A Tutorial}},
	author={Barmish, B Ross and Formentin, Simone and Hsieh, Chung-Han and Proskurnikov, Anton V and Warnick, Sean},
	booktitle={Proceedings of the IEEE Conference on Decision and Control (CDC)},
	pages={7441--7457},
	year={2024}
}

@article{ames2017control,
  author  = {Ames, Aaron D. and Xu, Xiangru and
             Grizzle, Jessy W. and Tabuada, Paulo},
  title   = {{Control Barrier Function Based Quadratic Programs
             for Safety Critical Systems}},
  journal = {IEEE Transactions on Automatic Control},
  volume  = {62},
  number  = {8},
  pages   = {3861--3876},
  year    = {2017}
}

@article{anevlavis2024controlled,
  author  = {Anevlavis, Tzanis and Liu, Zexiang and
             Ozay, Necmiye and Tabuada, Paulo},
  title   = {{Controlled Invariant Sets: Implicit Closed-Form
             Representations and Applications}},
  journal = {IEEE Transactions on Automatic Control},
  volume  = {69},
  number  = {7},
  pages   = {4506--4521},
  year    = {2024},
  doi     = {10.1109/TAC.2023.3336819}
}

@article{Barmish_Primbs_2015,
  title     = {{On a New Paradigm for Stock Trading via a Model-Free Feedback Controller}},
  author    = {Barmish, B Ross and Primbs, James A},
  journal   = {IEEE Transactions on Automatic Control},
  volume    = {61},
  number    = {3},
  pages     = {662--676},
  year      = {2015},
  publisher = {IEEE}
}

@book{beck2017first,
  author    = {Beck, Amir},
  title     = {{First-Order Methods in Optimization}},
  publisher = {SIAM},
  address   = {Philadelphia, PA},
  year      = {2017}
}

@article{blanchini1999set,
  author  = {Blanchini, Franco},
  title   = {{Set Invariance in Control}},
  journal = {Automatica},
  volume  = {35},
  number  = {11},
  pages   = {1747--1767},
  year    = {1999}
}

@book{blanchini2015set,
  author    = {Blanchini, Franco and Miani, Stefano},
  title     = {{Set-Theoretic Methods in Control}},
  edition   = {2nd},
  publisher = {Birkh{\"a}user},
  address   = {Cham},
  year      = {2015},
  doi       = {10.1007/978-3-319-17933-9}
}

@incollection{chekhlov2004portfolio,
  title     = {{Portfolio Optimization with Drawdown Constraints}},
  author    = {Chekhlov, Alexei and Uryasev, Stanislav and Zabarankin, Michael},
  booktitle = {Supply Chain and Finance},
  pages     = {209--228},
  year      = {2004},
  publisher = {World Scientific}
}

@article{chekhlov2005drawdown,
  title     = {{Drawdown Measure in Portfolio Optimization}},
  author    = {Chekhlov, Alexei and Uryasev, Stanislav and Zabarankin, Michael},
  journal   = {International Journal of Theoretical and Applied Finance},
  volume    = {8},
  number    = {1},
  pages     = {13--58},
  year      = {2005},
  publisher = {World Scientific}
}

@article{cherny2013portfolio,
  author  = {Cherny, Vladimir and Ob{\l}{\'o}j, Jan},
  title   = {{Portfolio Optimisation under Non-Linear Drawdown Constraints in a Semimartingale Financial Model}},
  journal = {Finance and Stochastics},
  volume  = {17},
  number  = {4},
  pages   = {771--800},
  year    = {2013}
}

@article{comelli2024inner,
  author  = {Comelli, Rom{\'a}n and Olaru, Sorin and
             Kofman, Ernesto},
  title   = {{Inner--Outer Approximation of Robust
             Control Invariant Sets}},
  journal = {Automatica},
  volume  = {159},
  pages   = {111350},
  year    = {2024},
  doi     = {10.1016/j.automatica.2023.111350}
}

@article{cvitanic1994portfolio,
  title   = {{On Portfolio Optimization under  ``Drawdown" Constraints}},
  author  = {Cvitani\'c, Jaksa and Karatzas, Ioannis},
  journal = {IMA Volumes in Mathematics and its Applications},
  volume  = {65},
  pages   = {35--46},
  year    = {1994}
}

@article{grossman1993optimal,
  title     = {{Optimal Investment Strategies for Controlling Drawdowns}},
  author    = {Grossman, Sanford J and Zhou, Zhongquan},
  journal   = {Mathematical Finance},
  volume    = {3},
  number    = {3},
  pages     = {241--276},
  year      = {1993},
  publisher = {Wiley Online Library}
}

@article{hernandez2023portfolio,
  author  = {Hern{\'a}ndez-Bustos, Daniel and Hern{\'a}ndez-Hern{\'a}ndez, Daniel},
  title   = {{Portfolio Management under Drawdown Constraint in Discrete-Time Financial Markets}},
  journal = {Journal of Applied Probability},
  volume  = {60},
  number  = {1},
  year    = {2023}
}

@inproceedings{hsieh2016kelly,
  title     = {{Kelly Betting Can be Too Conservative}},
  author    = {Hsieh, Chung-Han and Barmish, B Ross and Gubner, John A},
  booktitle = {Proceedings of the IEEE Conference on Decision and Control (CDC)},
  pages     = {3695--3701},
  year      = {2016}
}

@article{hsieh2017drawdown,
  title     = {{On Drawdown-Modulated Feedback Control in Stock Trading}},
  author    = {Hsieh, Chung-Han and Barmish, B Ross},
  journal   = {IFAC-PapersOnLine},
  volume    = {50},
  number    = {1},
  pages     = {952--958},
  year      = {2017},
  publisher = {Elsevier}
}

@inproceedings{hsieh2017inefficiency,
  title     = {{On Inefficiency of Markowitz-Style Investment Strategies When Drawdown is Important}},
  author    = {Hsieh, Chung-Han and Barmish, B Ross},
  booktitle = {Proceedings of the IEEE Conference on Decision and Control (CDC)},
  pages     = {3075--3080},
  year      = {2017}
}

@article{hsieh2023data,
  author  = {Hsieh, Chung-Han},
  title   = {{On Data-Driven Drawdown Control with Restart Mechanism in Trading}},
  journal = {IFAC-PapersOnLine},
  volume  = {56},
  number  = {2},
  year    = {2023}
}

@article{Ismail_2004,
  title     = {{On the Maximum Drawdown of a Brownian Motion}},
  author    = {Magdon-Ismail, Malik and Atiya, Amir F and Pratap, Amrit and Abu-Mostafa, Yaser S},
  journal   = {Journal of Applied Probability},
  volume    = {41},
  number    = {1},
  pages     = {147--161},
  year      = {2004},
  publisher = {Cambridge University Press}
}

@article{calafiore2008multi,
  title={{Multi-Period Portfolio Optimization with Linear Control Policies}},
  author={Calafiore, Giuseppe Carlo},
  journal={Automatica},
  volume={44},
  number={10},
  pages={2463--2473},
  year={2008},
  publisher={Elsevier}
}

@article{calafiore2013direct,
  title={{Direct Data-Driven Portfolio Optimization with Guaranteed Shortfall Probability}},
  author={Calafiore, Giuseppe Carlo},
  journal={Automatica},
  volume={49},
  number={2},
  pages={370--380},
  year={2013},
  publisher={Elsevier}
}

@book{maclean2011kelly,
  title={{The Kelly Capital Growth Investment Criterion: Theory and Practice}},
  author={MacLean, Leonard C and Thorp, Edward O and Ziemba, William T},
  volume={3},
  year={2011},
  publisher={World Scientific}
}

@article{hsieh2023asymptotic,
  title={{On Asymptotic Log-Optimal Portfolio Optimization}},
  author={Hsieh, Chung-Han},
  journal={Automatica},
  volume={151},
  pages={110901},
  year={2023},
  publisher={Elsevier}
}

@article{kardaras2017numeraire,
  author  = {Kardaras, Constantinos and Ob{\l}{\'o}j, Jan and Platen, Eckhard},
  title   = {{The Num\'eraire Property and Long-Term Growth Optimality for Drawdown-Constrained Investments}},
  journal = {Mathematical Finance},
  volume  = {27},
  number  = {1},
  pages   = {68--95},
  year    = {2017}
}

@incollection{Kelly_1956,
  title     = {{A New Interpretation of Information Rate}},
  author    = {Kelly Jr, John L},
  booktitle = {The Kelly Capital Growth Investment Criterion: Theory and Practice},
  pages     = {25--34},
  year      = {2011},
  publisher = {World Scientific}
}

@article{klass2005grossman,
  author  = {Klass, Michael J. and Nowicki, Krzysztof},
  title   = {{The {G}rossman and {Z}hou Investment Strategy Is Not Always Optimal}},
  journal = {Statistics \& Probability Letters},
  volume  = {74},
  number  = {3},
  pages   = {245--252},
  year    = {2005}
}

@article{konno1991mean,
  title={{Mean-Absolute Deviation Portfolio Optimization Model and its Applications to Tokyo Stock Market}},
  author={Konno, Hiroshi and Yamazaki, Hiroaki},
  journal={Management Science},
  volume={37},
  number={5},
  pages={519--531},
  year={1991},
  publisher={INFORMS}
}

@book{luenberger2013investment,
  author    = {Luenberger, David G.},
  title     = {{Investment Science}},
  edition   = {2nd},
  publisher = {Oxford University Press},
  year      = {2013},
  address   = {New York, NY},
  isbn      = {978-0-19-974008-6}
}

@inproceedings{Malekpour_Barmish_2012,
  title     = {{How Useful are Mean-Variance Considerations in Stock Trading via Feedback Control?}},
  author    = {Malekpour, Shirzad and Barmish, B Ross},
  booktitle = {Proccedings of the IEEE Conference on Decision and Control (CDC)},
  pages     = {2110--2115},
  year      = {2012}
}

@article{Markowitz_1952,
  title   = {{Portfolio Selection}},
  author  = {Markowitz, Harry M},
  journal = {The Journal of Finance},
  volume  = {7},
  pages   = {77--91},
  year    = {1952}
}

@article{nystrup2019multi,
  title     = {{Multi-Period Portfolio Selection with Drawdown Control}},
  author    = {Nystrup, Peter and Boyd, Stephen and Lindstr{\"o}m, Erik and Madsen, Henrik},
  journal   = {Annals of Operations Research},
  volume    = {282},
  number    = {1},
  pages     = {245--271},
  year      = {2019},
  publisher = {Springer}
}

\end{document}